\documentclass{amsart}

\usepackage{amsmath,amssymb}
\usepackage{graphicx}
\usepackage{epstopdf}

\newcommand{\arr}{\longrightarrow}

\newcommand{\mapdown}[1]%
{\Big\downarrow\rlap{$\vcenter{\hbox{$\scriptstyle#1$}}$}}

\newcommand{\edge}[1]{\stackrel{#1}{\rule[3pt]{2em}{.5pt}}}

\newcommand{\one}{\mathsf{1}}
\newcommand{\zero}{\mathsf{0}}

\newcommand{\Z}{\mathbb{Z}}
\newcommand{\G}{\Gamma}

\newcommand{\X}{\mathcal{X}}

\newcommand{\be}{\mathsf{s}}
\newcommand{\en}{\mathsf{r}}

\newtheorem{theorem}{Theorem}[section]

\newtheorem{proposition}[theorem]{Proposition}

\theoremstyle{definition}
\newtheorem{defi}[theorem]{Definition}
\newtheorem{example}[theorem]{Example}

\title{Substitutional subshifts and growth of groups}
\author{Volodymyr Nekrashevych}

\begin{document}
\maketitle

\begin{abstract}
We show how to use symbolic dynamics of Schreier graphs to embed the Grigorchuk group into a simple torsion group of intermediate growth and to construct uncountably many growth types of simple torsion groups.
\end{abstract}

\section{Introduction}

The paper is a continuation of~\cite{nek:burnside}, where first examples of simple groups of intermediate growth were constructed. It has two main goals: illustrating flexibility of the techniques introduced in~\cite{nek:burnside} and~\cite{nek:fullgr}, and describing a different approach for defining and studying groups introduced there. We use explicit symbolic substitutional systems and Schreier graphs rather than generate groups by homeomorphisms of the Cantor set. The difference is mostly formal (as there are standard ways of transforming one type of definitions into the other) but one approach may be more convenient in some situations than the other.

As the first illustration, we show how to embed the first Grigorchuk group~\cite{grigorchuk:80_en} into a simple torsion group of intermediate growth. A procedure of constructing a finitely generated simple group from an expansive minimal action on a Cantor set is derscribed in~\cite{nek:fullgr}. The standard action of the Grigorchuk group on the Cantor set is equicontinuous (hence not expansive) which is related to the fact that the Grigorchuk group is residually finite. However, we can easily transform an action into an expansive one by ``exploding'' points. This is, for example, the way how equicontinuous action of $\Z$ on the circle by irrational rotation is transformed into an expansive \emph{Denjoy system}, see~\cite{nielsen:denjoy,denjoy:curbes}. One can perform the   same trick with the Grigorchuk group and get a simple finitely generated group containing the Grigorchuk group and sharing with it many finiteness properties.

We describe this construction in a symbolic way by describing how to construct the Schreier graphs of the new group by substitutions. In fact, the Schreier graphs of the virtually simple group containing the Grigorchuk group will coincide with the graphs of the action of the Grigorchuk group on the Cantor set, except that one of the generators will be split into three elements with disjoint supports, so that the edges labeled by one of the generators of the Grigorchuk group will be labeled by one of three different labels.

We give a direct proof that the new group is virtually simple just by studying the Schreier graphs. The fact that the group is torsion and of intermediate growth will follow from the results of~\cite{nek:burnside}.

L.~Bartholdi and A.~Erschler have proved in~\cite{bartholdiErschler:imbeddings} that every group of locally sub-exponential growth can be embedded into a 2-generated group of sub-exponential growth. It would be interesting to know if every such group can be embedded into a simple group of intermediate growth.

Our second result is showing that there are uncountably many pairwise different growth types among finitely generated simple groups. We construct a family of groups $G_\alpha$ generated by 6 elements such that each group in the family has a simple subgroup of index 16, and there are uncountably many growth types of of groups in the family. The idea of the proof is similar to the analogous result of R.~Grigorchuk in~\cite{grigorchuk:growth_en} for the class of finitely generated groups. But it has to use a different approach to defining groups, since the approach in~\cite{grigorchuk:growth_en} is based on defining group actions on rooted trees, which necessarily leads to residually finite groups.

Our result is the first explicit example of uncountably many growth types of simple groups. It was shown  in~\cite{minasyanosinwitzel} that there are uncountably many pairwise non-quasi-isometric finitely generated groups that are simple (among many other properties). 
Examples of infinite sets of pairwise non-quasi-isometric simple finitely presented groups are constructed in~\cite{capraceremy:building,SkipperZaremsky}.

\section{Groups defined by Schreier graphs}

\subsection{Graphs of actions}
Let $S$ be  a finite set. An \emph{$S$-labeled graph} $\G$ is given by the \emph{set of vertices} $V$, the \emph{set of edges (arrows)} $E$, the \emph{source and range} maps $\be, \en:E\arr V$, and a \emph{labeling map} $\lambda:E\arr S$.
The graph is \emph{perfectly labeled} if for every $v\in V$ and every $h\in S$ there exists a unique $e_1\in E$ such that $\be(e_1)=v$ and $\lambda(e_1)=h$, and a unique $e_2\in E$ such that $\en(e_2)=v$ and $\lambda(e_2)=h$.

A \emph{morphism} between two graphs $\G_1=(V_1, E_1, \be, \en, \lambda)$ and $\G_2=(V_2, E_2, \be, \en, \lambda)$ is a pair of maps $\phi:V_1\arr V_2, \phi:E_1\arr E_2$ such that $\be(\phi(e))=\phi(\be(e))$, $\en(\phi(e))=\phi(\en(e))$, and $\lambda(\phi(e))=\lambda(e)$ for every $e\in E_1$. The morphism is an \emph{isomorphism} if the maps $\phi:V_1\arr V_2$ and $\phi:E_1\arr E_2$ are bijections. A morphism is a \emph{covering} if it is surjective and for every $v\in V_1$ the map $\phi$ induces a bijection from $\be^{-1}(v)$ to $\be^{-1}(\phi(v))$ and from $\en^{-1}(v)$ to $\en^{-1}(\phi(v))$.

Note that if $\G_1$ and $\G_2$ are perfectly labeled and connected, then every morphism between them is a covering. Moreover, for any two perfectly labeled connected graphs $\G_1$ and $\G_2$ and every pair of vertices $v_i\in\G_i$ there exists at most one morphism $\phi:\G_1\arr\G_2$ such that $\phi(v_1)=v_2$.

If $\G$ is perfectly labeled, then every $h\in S$ defines a permutation $V\arr V$ by the condition $h(v)=\en(e)$, where $e$ is the unique edge such that $\be(e)=v$ and $\lambda(e)=h$.
The group generated by these permutations is called the \emph{group defined by $\G$}. We usually identify the elements of $S$ with the corresponding permutations of $V$. Conversely, for every group generated by a finite set $S$ of permutations of a set $V$ we can define the corresponding \emph{graph of the action} with the set of vertices $V$, set of edges $E=S\times V$, and maps $\be(s, v)=v$, $\en(s, v)=s(v)$, $\lambda(s, v)=s$.

All our actions are from the left.

For a graph $\G$ the combinatorial distance between two vertices $v_1, v_2$ is the smallest number of edges in a sequence $e_1, e_2, \ldots, e_n$ such that $v_1\in\{\be(e_1), \en(e_1)\}$, $v_2\in\{\be(e_n), \en(e_n)\}$, and $\{\be(e_i), \en(e_i)\}\cap \{\be(e_{i+1}), \en(e_{i+1}\}\ne\emptyset$. In other words, it is the length of the shortest path connecting $v_1$ to $v_2$ if we disregard the orientation of the edges.

A \emph{rooted graph} is a graph with a marked vertex called the \emph{root}. A morphism of two rooted graphs is a morphism mapping the root to the root. We have already mentioned that two rooted perfectly labeled graphs have at most one morphism between them.

Let $\mathcal{G}_S$ be the set of \emph{rooted} perfectly $S$-labeled connected graphs. Denote by $B_v(R)$ the ball of radius $R$ with center in $v$. Introduce a metric on $\mathcal{G}_S$ by defining the distance between two rooted graphs $(\G_1, v_1)$, $(\G_2, v_2)$ to be $2^{-R}$, where $R$ is the largest radius such that the rooted graphs $(B_{v_1}(R), v_1)$ and $(B_{v_2}(R), v_2)$ are isomorphic (as rooted labeled oriented graphs). This metric induces a natural topology on $\mathcal{G}_S$. An equivalent definition of $\mathcal{G}_S$ is as the Chabauty space of all subgroups of the free group $F_S$ generated by $S$. Namely, a rooted perfectly $S$-labeled graph $(\G, v)$ is in a bijective correspondence with the subgroup of $F_S$ consisting of products $s_1s_2\cdots s_n\in F_S$ such that the corresponding product $s_1s_2\ldots s_n$ of permutations of $V$  fixes $v$. (The graph $(\G, v)$ is reconstructed as the \emph{Schreier graph} of the cosets modulo the subgroup.) The topology on $\mathcal{G}_S$ is then the same as the topology on the space of subgroups of $F_S$ induced from the direct product topology on the set $2^{F_S}$ of subsets of $F_S$.

Suppose that $G$ is the group defined by a perfectly $S$-labeled connected graph $\G$. If $\Delta$ is another connected perfectly $S$-labeled graph such that there is a morphism $\phi:\G\arr\Delta$, then the action of $G$ on $\G$ pushes forward  by $\phi$ to an action on $\Delta$. In particular, the group generated by $\Delta$ is a quotient of the group generated by $\G$ (and the corresponding epimorphism is induced by the tautological map on $S$).

\emph{The hull} of a perfectly labeled graph $\G$ is the closure of the set $\{(\G, v)\;:\;v\in V\}$ in the space $\mathcal{G}_S$, where $V$ is the set of vertices of $\G$. 
It is easy to see that if a graph $\Delta$ is an element of the hull of a graph $\G$, then the group defined by $\Delta$ is a quotient of the group defined by $\G$ (also with the epimorphism induced by the tautological map on $S$). In particular, if $\G_1$ and $\G_2$ have the same hulls, then they define the same groups.

Another easy observation, which we will often use, is that if $g\in G$ is an element of length $n$ (with respect to the generating set $S$) of the group defined by $\G$, then the image $g(v)$ of a vertex $v\in\G$ belongs to $B_v(n)$ and depends only on the isomorphism class of the rooted graph $(B_v(n), v)$. Namely, if $g=h_1h_2\ldots h_n$, then $g(v)$ is the end of the unique path of edges $e_n, e_{n-1}, \ldots, e_1$ starting in $v$ in which the edge $e_i$ is labeled by $h_i$. Here we assume that each edge $e$ comes in two incarnations: in its original orientation labeled by $h\in S$, and in the opposite orientation $e^{-1}$ labeled by $h^{-1}$ (with $\be(e^{-1})=\en(e)$, $\en(e^{-1})=\be(e)$).

\subsection{Linear graphs}

Most of graphs defining group in our paper will be of a very special form.

We say that a sequence $w=(A_n)_{n\in\Z}$ of non-empty subsets of $S$ is \emph{admissible} if $A_n\cap A_{n+1}=\emptyset$ for every $n\in\Z$. We  imagine such a sequence as a graph with the set of vertices $\Z$ in which for every $n$ there are $|A_n|$ edges connecting $n$ to $n+1$ labeled by the elements of $A_n$. We also add loops at $n$ labeled by the elements of $S\setminus (A_n\cup A_{n-1})$. Then the graph $\G_w$ associated with the sequence $w$ is perfectly $S$-labeled. Note that all the edges are not directed in our case. (More precisely, we assume that $h=h^{-1}$ for every $h\in S$, so that each edge $e$ is equal to its inverse.)

Suppose that $w$ is an admissible sequence. Then the action of an element $h\in S$ on the set of vertices $\Z$ of $\G_w$ is given by 
\[h(n)=\left\{\begin{array}{rl} n+1 & \text{if $h\in A_n$,}\\ n-1 & \text{if $h\in A_{n-1}$,}\\ n & \text{otherwise.}\end{array}\right.\]
Denote by $G_w$ the group defined by the corresponding perfectly labeled graph, i.e., the group of permutations of $\Z$ generated by the permutations  defined by the above formula. It is a subgroup of the \emph{wobbling group} on $\Z$, see~\cite{juschenkomonod}.

We assume that $\G_w$ is a rooted graph with the root $0\in\Z$.
The hull of $\G_w$ can be described as the set of rooted graphs $\G_u$ for all sequences $u$ belonging to the \emph{subshift generated by $w$} and the sequence $w^{-1}$ written in the opposite direction, i.e., to the set of all sequences $u$ such that every finite subword of $u$ is a subword of $w$.or of $w^{-1}$. Equivalently, it is the closure of the union of the orbits of $w$ and $w^{-1}$ under the shift. If the sets of finite subwords of $w$ and of $w^\top$ are equal, then the hull is equal to the subshift generated by $w$. (This will be the case in all of our examples.)

\begin{example}
Take $S=\{a, b\}$, and let $w=(\ldots\{a\}\{b\}\{a\}\{b\}\ldots)$. Then $G_w$ is the infinite dihedral group. The subshift generated by $w$ has two elements.
\end{example}

\begin{example}
Consider $S=\{a, b, c\}$ and the Markov chain $w=(\{s_n\})_{n\in\Z}$, where $s_i$ is chosen at random with probability 1/2 from the two possibilities in $S\setminus\{s_{i-1}\}$. Then $G_w$ is isomorphic with probability 1 to the free product $C_2*C_2*C_2$, since every element of the free product will appear as a subword of $w$ with probability 1. In other words, with probability 1 the subshift generated by $w$ is the set of all sequences in the Markov chain.
\end{example}

On the other hand, the group $G_w$ is amenable for a wide class of sequences $w$.

\begin{defi}
A sequence $(a_n)_{n\in\Z}$ is \emph{repetitive} if for every finite segment $(a_i, a_{i+1}, \ldots, a_{i+n})$ there exists a constant $C_n$ such that for every $k\in\Z$ there exists $j\in\Z$ such that $|j-k|<C_n$ and $(a_i, a_{i+1}, \ldots, a_{i+n})=(a_j, a_{j+1}, \ldots, a_{j+n})$. We say that it is \emph{linearly repetitive} if there exists a constant $L$ such that $C_n\le Ln$ for all $n\ge 1$.
\end{defi}

A subshift (i.e., a closed shift-invariant subset of the space of sequences) is called minimal if the orbit of every its element is dense.
It is a classical fact that a sequence is repetitive if and only if it generates a \emph{minimal subshift}. 

We say that a subshift $\mathcal{S}_w$ is linearly repetitive, if $w$ is linearly repetitive. It is easy to check that this does not depend on the choice of a particular sequence $w$ generating the subshift.

The following is a direct corollary of the main result of~\cite{juschenkomonod}.

\begin{theorem}
If $w=(A_n)_{n\in\Z}\in\left(2^S\right)^\Z$ is a repetitive admissible sequence, then $G_w$ is amenable.
\end{theorem}

A \emph{segment} or a \emph{subword} of a sequence $(A_n)$ is any finite sequence $(A_i, A_{i+1}, \ldots, A_{i+n})$. A \emph{one-sided sequence} is sequence of the form $(A_0, A_1, \ldots)$ or $(\ldots, A_{-1}, A_0)$. When we talk about segments we usually interpret them as graphs.

For a set $\mathcal{W}$ of finite or infinite (one-sided or two-sided) admissible sequences, we define the \emph{subshift generated by $\mathcal{W}$} as the set of all sequences $(A_n)_{n\in\Z}$ such that every finite segment $(A_i, A_{i+1}, \ldots, A_{i+k})$ is a segment of an element of $\mathcal{W}$.

\subsection{Substitutions}

We will construct admissible sequences over alphabet $S$ using the following notion of a \emph{substitution}.

We start with a set $X_0$ of \emph{initial segments}, a set $C$ of segments called \emph{connectors}, and define inductively the \emph{$n$th generation segments} equal to $x=x_1^{\epsilon_1}e_1x_2^{\epsilon_2}e_2\ldots e_{k-1}x_k^{\epsilon_k}$, where $x_i$ belong to the set $X_{n-1}$ of segments of the $(n-1)$st generation, $e_i\in C$, and $\epsilon_i$ is either nothing or $-1$, denoting the operation of writing a sequence in the opposite direction. The number $k$ and the elements $x_i$, $e_i$, $\epsilon_i$ depend on $x$. 

Our recurrent rules (which we will call \emph{substitutions}) may be \emph{stationary}, i.e., each set $X_n$ is identified with the same set $X$, and the rules of creating the next generation of segments from the segments of the previous one  do not depend on $n$.

We use connectors only for convenience. One can instead replace the set $X_n$ of segments of $n$th generation by $X_n\cup C$ (or by $X_n\cup X_nC$). Then the next generation of segments is obtained just by concatenation of segments of the previous generation.

\begin{example}
\label{ex:Grigorchukgroupgraphs}
Let $S=\{a, b, c, d\}$. The single initial segment will be $\{a\}$. The sets $X_n$ will consist of a single segment $I_n$, and the rules are
\[I_{n+1}=\left\{\begin{array}{rl} I_n\{b, c\}I_n & \text{if $n\equiv 1\pmod{3}$,}\\
I_n\{b, d\}I_n & \text{if $n\equiv 2\pmod{3}$,}\\
I_n\{c, d\}I_n & \text{if $n\equiv 3\pmod{3}$.}\end{array}\right.\]

This rule is not stationary, but if we keep only the sets $X_0, X_3, X_6, \ldots$, then the rule will be a stationary substitution.
\end{example}

\begin{example}
\label{ex:myfibonacci}
Take $S=\{a_i, b_i, c_i, d_i\;:\;i=0, 1, 2,3\}$, the initial segments $I_0$, $I_1$ consisting of single vertices, connectors $e_0=\{a_0, b_0, c_0\}$, $e_1=\{a_1, b_2, c_1\}$, $e_2=\{a_2, b_2, c_2\}$, and $e_{3k+i}$ for $k>0$ equal to $\{b_i, c_i\}$ for $k\equiv 0\pmod{3}$, $\{b_i, d_i\}$ for $k\equiv 1\pmod{3}$, $\{c_i, d_i\}$ for $k\equiv 2\pmod{3}$. Define the segments $I_n$ for $n\ge 2$ by the substitution
\[I_n=I_{n-2}^{-1}e_{n-2}I_{n-1}^{-1}.\]
The group defined by these Schreier graphs in the group of intermediate growth analyzed in~\cite[Section~8]{nek:burnside}.
\end{example}

The following two propositions are classical  (see, for example~\cite[Proposition~1.4.6]{bertherigo:combinatoricsautomata}).

\begin{proposition}
Let $C$ be a finite set of segments, and let $X_n$ be a sequence of finite sets of segments such that each segment $x\in X_{n+1}$ is of the form $x_1e_1x_2e_2\ldots e_{m_x-1}x_{m_x}$, for $x_i\in X_n$ and $e_i\in C$. Suppose that for every $n$ there exists $m$ such that every segment $z\in X_m$ contains every segment $x\in X_n$. Then the subshift $\mathcal{S}$ defined by the set $\bigcup_n X_n$ is minimal. 
\end{proposition}

\begin{proof}
By definition, for every $n$ every element of $\mathcal{S}$ is of the form $\ldots e_{-1}x_0e_1x_1e_2\ldots$ for some sequences $x_i\in X_n$ and $e_i\in C$. If $I$ is any segment of an element of $\mathcal{S}$, then, by definition of $\mathcal{S}$, it is a subsegment of a segment $x\in X_n$ for some $n$. Let $m\ge n$ be such that every segment $y\in X_m$ contains $x$. Then the distance between any two consecutive copies of $x$ in every $w\in\mathcal{S}$ will be less than twice the maximum of length of elements of $X_m$ plus the maximum of lengths of elements of $C$. It follows that every $w\in\mathcal{S}$ is repetitive.
\end{proof}

\begin{proposition}
\label{pr:linrepetitive}
Let $X_n$ be a sequence of finite sets of segments and let $\mathcal{S}$ be the subshift generated by their union.

Suppose that the following conditions are satisfied
\begin{enumerate}
\item Every element  $x=X_{n+1}$ is a concatenation $x_1x_2\ldots x_m$ of elements of $X_n$.  
\item There exists $k$ such that if $x_1x_2$, for $x_1, x_2\in X_n$, is a segment of an element of $\mathcal{S}$, then $x_1x_2$ is a segment of an element of $X_{n+k}$.
\item The numbers $m$, $|X_n|$ in condition (1) are uniformly bounded.
\end{enumerate}
Then the subshift generated by the union of the sets $X_n$ is linearly repetitive.
\end{proposition}

\begin{proof}
After replacing the sequence $X_n$ by $X_{nk}$, we may assume that $k=1$. Note that then every segment $x\in X_n$ is a sub-segment of every element of $X_{n+1}$.
Let $L_n$ be the maximal length of an element of $X_n$. Then the length of every element of $X_{n+1}$ is at least $|X_n|L_n$ and at most $mL_n$. It follows that the ratio of lengths of any two elements of $X_{n+1}\cup X_n$ belongs to an interval $(C^{-1}, C)$ for some constant $C>1$. 

Let $v$ be a segment of an element of the sushift $\mathcal{S}$ generated by the union of the sets $X_n$. Let $n$ be the smallest number such that $v$ is a subsegment of a segment $x_1x_2$ for $x_1, x_2\in X_n$ such that $x_1x_2$ is a segment of an element of $\mathcal{S}$. Then the length of $v$ is not greater than $2L_n$. On the other hand, the length of $v$ can not be smaller than the length of the shortest element of $X_{n-1}$. It follows that the ratio of the lengths of $v$ and any element of $X_n$ belongs to an interval $(C_1^{-1}, C_1)$ for a constant $C_1>1$ not depending on $v$ and $n$. Then the second condition of the proposition implies that the gaps between isomorphic copies of $v$ in an element of $\mathcal{S}$ have length bounded above by $C_2|v|$.
\end{proof}

\subsection{Groups of intermediate growth}

If $G$ is a group generated by a finite set $S$, then the associated \emph{growth function} is the number $\gamma_G(n)=\gamma_{G, S}(n)$ of elements of $G$ that can be written as products $s_1s_2\ldots s_k$ for $s_i\in S\cup S^{-1}$ and $k\ge n$. The growth \emph{rate} is the equivalence class of $\gamma_{G, S}(n)$ with respect to the equivalence relation identifying two non-decreasing functions $f_1, f_2$ if there exists a constant $C>1$ such that 
\[f_1(n)\le f_2(Cn),\quad\text{and}\quad f_2(n)\le f_1(Cn)\]
for all $n\ge 1$. The growth rate does not depend on the choice of the generating set.

\begin{theorem}
\label{th:intermediategrowth}
Let $\mathcal{S}\subset\left(2^S\right)^{\Z}$ be a linearly repetitive infinite subshift consisting of admissible sequences and containing three sequences $(B_n)_{n\in\Z}$, $(C_n)_{n\in\Z}$, $(D_n)_{n\in\Z}$ such that $B_n=B_{-n}=C_n=C_{-n}=D_n=D_{-n}$ for all $n\ge 1$ and $B_0=\{c, d\}, C_0=\{b, d\}, D_0=\{b, c\}$ for some $b, c, d\in S$.

Then $G_{\mathcal{S}}$ is a infinite torsion group and its growth function $\gamma(R)$ is bounded from above by $\exp(C_1n/\exp(C_2\sqrt{\log n}))$ for some constants $C_1, C_2$.
\end{theorem}

\begin{proof}
The proof is the same as the proof of~\cite[Theorem~6.6]{nek:burnside}, where only linear repetitivity and the structure of the graphs $(B_n)_{n\in\Z}$, $(C_n)_{n\in\Z}$ and $(D_n)_{n\in\Z}$ and their smallest common covering graph $\Xi$ are used.
\end{proof}

\begin{example}
The group defined in Example~\ref{ex:Grigorchukgroupgraphs} is of intermediate growth, since the corresponding graphs are linearly repetitive by Proposition~\ref{pr:linrepetitive}, and the corresponding subshift contains $I_\infty^{-1}\{b, c\}I_\infty, I_\infty^{-1}\{d, b\}I_\infty$, and $I_\infty^{-1}\{c, d\}I_\infty$, where $I_\infty=\{a\}\{b, c\}\{a\}\{d, b\}\{a\}\{b, c\}\{a\}\ldots$ is the inductive limit of the segments $I_n$ with respect to the embeddings of $I_n$ to the suffix of $I_{n+1}$.
\end{example}

\begin{example}
The same arguments show that the group defined in Example~\ref{ex:myfibonacci} is also of intermediate growth.
\end{example}

\section{Embedding Grigorchuk group into a simple group}
\label{s:simpleGrigorchuk}

The \emph{(first) Grigorchuk group} $G$ is the group generated by the transformations of $\{\zero, \one\}^\infty$ generated by the following recurrently defined permutations $a, b, c, d$
\begin{alignat*}{2}
a(\zero w)&=\one w, &\qquad  a(\one w)&=\zero w,\\
b(\zero w)&=\zero a(w), &\qquad  b(\one w)&=\one  c(w),\\
c(\zero w)&=\zero a(w), &\qquad b(\one w)&=\one  d(w),\\
d(\zero w)&=\zero w, &\qquad  d(\one w)&=\one  b(w).
\end{alignat*}

The following recurrent description of the orbital graphs of the action of $G$ on $\{\zero, \one\}^\infty$ is well known, see~\cite{bgr:spec}. The precise description of the space of orbital graphs is from~\cite{vorob:schreiergraphs}.

\begin{proposition}
Let $\mathcal{S}$ be the subshift defined by following segments:
\[I_1=\{a\}, \qquad I_{n+1}=I_n e_n I_n,\]
where $e_n=\left\{\begin{array}{rl}\{b, c\} & \text{for $n\equiv 1\pmod{3}$,}\\
\{d, b\} & \text{for $n\equiv 2\pmod{3}$,}\\
\{c, d\} & \text{for $n\equiv 3\pmod{3}$.}\end{array}\right.$
Then the group defined by $\mathcal{S}$ is isomorphic to the Grigorchuk group $G$ (for the same generators $a, b, c, d$).

There is a surjective equivariant continuous map $\Phi:\mathcal{S}\arr\{\zero, \one\}^\infty$, which is one-to-one except for the set of graphs isomorphic (as non-rooted graphs) to one of the three graphs $I_\infty^{-1}e_nI_\infty$, where $I_\infty$ is the direct limit of the segments $I_n$ with respect to the embedding of $I_n$ into the left half of $I_{n+1}$. The map $\Phi$ maps all three rooted graphs $I_\infty^{-1}e_nI_\infty$ to $\one\one\one\ldots$, so that $\Phi$ is three-to-one on the exceptional set. 

The orbital graph of $\one\one\one\ldots$ is one-ended chain described by the sequence $I_\infty$.
\end{proposition}

We see that Example~\ref{ex:Grigorchukgroupgraphs} defined the orbital graphs of the Grigorchuk group.
We refer the readers to~\cite{vorob:schreiergraphs} for the proof of the proposition. Here we will only describe the map $\Phi$ by labeling the vertices of $I_n$. Define $I_1$ as the graph $\one\edge{a}\zero$. Define then \[I_{n+1}=I_n\one e_n (I_n\zero)^{-1},\] where $I_n\one$ and $I_n\zero$ are obtained from $I_n$ by appending to the end of the names of their vertices symbols $\one$ and $\zero$, respectively. As usual, $I^{-1}$ denotes reverting the orientation of a segment $I$.

The new vertex-labeled graphs $I_n$ are isomorphic to the graphs $I_n$ from the proposition. (Note that $I_n$ in the proposition are symmetric.)

It is checked by induction that the left endpoint of $I_n$ is $\underbrace{\one\one\ldots\one}_{\text{$n$ times}}$, its right endpoint is $\underbrace{\one\one\ldots\one}_{\text{$n-1$ times}}\zero$, and if $v$ and $u$ are two edges of $I_n$ connected by an edge labeled by a generator $s$, then $s(vw)=uw$ for all $w\in\{\zero, \one\}^\infty$. 
The map $\Phi$ is then the limit of the described labeling.

The Grigorchuk group also acts on the set $\{\zero, \one\}^*$ of finite words (using the same definition as for the infinite sequences). The action preserves the structure of a rooted tree on $\{\zero, \one\}^*$ (where a vertex $v$ is connected to $vx$ for every $x\in\{\zero, \one\}$).
This implies that the Grigorchuk group is residually finite.

Dynamically, the fact that the Grigorchuk group is residually finite corresponds to the fact that its action on $\{\zero, \one\}^\infty$ is equicontinuous. On the other hand, expansive actions can be used to construct simple groups, see~\cite{nek:fullgr}. A standard trick to force expansivity is ``exploding'' points of an orbit. For example, this is the way expansive Denjoy homeomorphisms of the Cantor set are obtained from equicontinuous actions by irrational rotations on the circle, see~\cite{nielsen:denjoy,denjoy:curbes}. 

We can perform the same trick with the Grigorchuk group. For example, we can split the point $\zero\zero\zero\ldots$ in two by separating the sequences with an odd and an even numbers of leading symbols $\zero$ into two neighborhoods of the two copies of $\zero\zero\zero\ldots$. 
After we propagate this split along the $G$-orbit of $\zero\zero\zero\ldots$, the Grigorchuk group will act on the new Cantor set expansively. Then the corresponding alternating full group of the action, as defined in~\cite{nek:fullgr}, will be simple and finitely generated.

More explicitly, the new Cantor set $\X$ will be the set of all right-infinite sequences $x_1x_2\ldots$ over the alphabet $\{\zero_0, \zero_1, \one\}$ such that each length 2 subword $x_ix_{i+1}$ belongs to the set
\[\{\zero_0\zero_1,\quad\zero_1\zero_0,\quad\zero_1\one,\quad\one\zero_0,\quad\one\zero_1,\quad\one\one\}.\]
Here the symbol $\zero_0$ ``predicts'' that there will be an even number of zeros (including itself) before the first $\one$, while $\zero_1$ ``predicts'' that the number of zeros will be odd.

The operation of erasing indices is then a continuous surjective map $\X\arr\{\zero, \one\}^\infty$, which is one-to-one except for the sequences $w\in\{\zero, \one\}^\infty$ containing finitely many symbols $\one$, which have two preimages.

The action of the Grigorchuk group on $\{\zero, \one\}^\infty$ naturally lifts to the action on $\X$ given by the rules
\[a(\one\zero_x w)=\zero_{1-x}\zero_x w,\quad a(\one\one w)=\zero_1\one w,\quad a(\zero_x w)=\one w,\]
and
\begin{alignat*}{2}
b(\zero_{1-x}\zero_x w)&=\zero_1 a(\zero_x w),&\quad b(\zero_1\one\zero_xw)&=\zero_xa(\one\zero_x w),\\
b(\zero_1\one\one w)&=\zero_0 a(\one\one w),&\quad b(\one w)&=\one c(w),\\
c(\zero_{1-x}\zero_x w)&=\zero_1 a(\zero_x w),& c(\zero_1\one\zero_xw)&=\zero_xa(\one\zero_x w),\\
c(\zero_1\one\one w)&=\zero_0 a(\one\one w),&\quad c(\one w)&=\one d(w),\\
d(\zero_x w)&=\zero_x w,&\quad d(\one w)&=\one b(w),
\end{alignat*}

One can show that this action is expansive. The topological full group of the action is generated by $b, c, d$, and the restrictions of $a$ to three subsets of $\X$ corresponding to the three possible values of the \emph{second} coordinate of a sequence $w\in\X$:
\[\one\zero_0w\stackrel{a_0}{\longleftrightarrow}\zero_1\zero_0w,\quad\one\zero_1w\stackrel{a_1}{\longleftrightarrow}\zero_0\zero_1w,\quad\one\one w\stackrel{a_2}{\longleftrightarrow}\zero_1\one w.\]

Then, by the results of~\cite{nek:fullgr} and~\cite{matui:etale}, the derived subgroup of the full group $\langle b, c, d, a_0, a_1, a_2\rangle$ is simple. It has finite index in the full group, since the full group is generated by finitely many elements of order 2.

Instead of proving the above statements, we will describe the full group $\widehat G$ anew as a group defined by its orbital graphs, and prove that it is virtually simple directly using only the structure of its Schreier graphs.

The generating set of our group $\widehat G$ will be $S=\{a_0, a_1, a_2, b, c, d\}$.
Consider the following initial segments
\begin{align*}
I_1 &=\{a_2\}\{b, c\}\{a_0\},\\
J_1 &=\{a_2\}\{b, c\}\{a_1\},
\end{align*}
and for $n\ge 1$ define
\begin{align*}
I_{n+1}&= J_ne_nJ_n^{-1},\\
J_{n+1}&=J_ne_n I_n^{-1},
\end{align*}
where the connector $e_n$ is equal to $\{b, c\}$, $\{d, b\}$, $\{c, d\}$ if $n\equiv 0, 1, 2\pmod{3}$, respectively.  Note that $I_n$ is symmetric for $n\ge 2$.

For example, we have
\[I_2=\{a_2\}\{b, c\}\{a_1\}\{d, b\}\{a_1\}\{b, c\}\{a_2\},\quad J_2=\{a_2\}\{b, c\}\{a_1\}\{d, b\}\{a_0\}\{b, c\}\{a_2\}.\]
We see that $I_n$ and $J_n$ for $n\ge 0$ start and end with $\{a_2\}\{b, c\}$ and $\{b, c\}\{a_2\}$, and that each of the edges $\{a_0\}$ and $\{a_1\}$ is surrounded by $\{d, b\}$ on one side and $\{b, c\}$ on the other.

Let $\mathcal{S}$ be the subshift generated by the set of segments $I_n$, $J_n$, and let $\widehat G$ be the group it defines. Note that $a_0, a_1, a_2$ commute (they are never neighbors) and that the subgroup $\langle a, b, c, d\rangle\le\widehat G$, where $a=a_0a_1a_2$, has the same orbital graphs as the Grigorchuk group, hence it is isomorphic to it.

\begin{figure}
\includegraphics{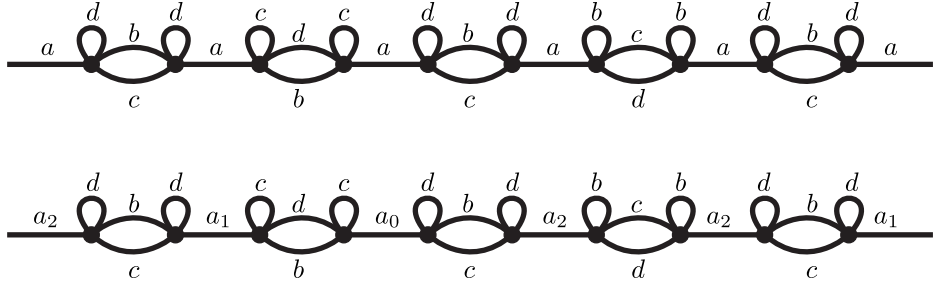}
\caption{Orbital graphs of $G$ and $\widehat G$}
\label{fig:schrgr}
\end{figure}

For every $n$ each orbital graph $\G_w$ is obtained by connecting infinitely many copies of $I_n$ and $J_n$ by connectors $e_0, e_1, e_2$ in some order. It follows 
from the description of $I_1$ and $J_1$ that the decomposition into copies of $I_1$ and $J_1$ is unique. This in turn implies that the decomposition of $\G_w$ into the copies of $I_n$ and $J_n$ is unique for every $n$. 

Let $H_n$, for $n\ge 2$, be the subgroup of elements $g\in\widehat G$ satisfying the following conditions:
\begin{enumerate}
\item Each copy of $I_n$ and $J_n$ in $\G_w$ is $g$-invariant.
\item The element $g$ commutes with isomorphisms between copies of $I_n$ and with isomorphisms between copies of $J_n$.
\item The action of $g$ on a copy of $I_n$ leaves both halves $J_{n-1}, J_{n-1}^{-1}$ of $I_n$ invariant.
\end{enumerate}

Denote by $S_{J_n}$ the symmetric group on the set of vertices of $J_n$.  It follows from the definition that $H_n$ is naturally identified with a subgroup of $S_{J_n}\times S_{J_{n-1}}$. Denote by $A_{J_n}$ the corresponding alternating subgroups.

\begin{proposition}
\label{pr:descriptionofHn}
The group $H_n$ contains $S_{J_n}\times S_{J_{n-1}}$ for every $n\ge 2$.
\end{proposition}

\begin{proof}
Denote by $H_{I_n}$ and $H_{J_n}$ the intersections of $H_n$ with $\{1\}\times S_{J_{n-1}}$ and $S_{J_n}\times \{1\}$, respectively.

It is easy to see that the element $a_0$ and the commutators $[b, a_0], [c, a_0], [d, a_0]$ belong to $H_{J_2}$. It follows that $H_{J_2}$ contains the symmetric group on the four vertices of the subsegment $\{d, b\}\{a_0\}\{b, c\}$ of $J_2$. Taking commutators of the elements of this symmetric group with $b, c, a_1, a_2$, and taking into account that $a_0\in H_{J_2}$, we conclude that $H_{J_2}=S_{J_2}\times\{1\}$.

We have $a_1\in H_2$. Since the projection of $H_2$ onto $S_{J_2}$ is surjective, it follows that $H_{I_2}$ contains the action of $a_1$ on the vertices of $I_2$. Taking commutators with $b, c, a_2$, we conclude that the elements of $H_{I_2}$ can permute the vertices of one half $\{a_2\}\{b, c\}\{a_1\}$ of $I_2$ by any permutation, i.e., that $H_{I_2}=S_{J_1}$. This finishes the proof for $n=2$.

Suppose that we know that proposition holds for $n$, and let us prove it for $n+1$.
Note that the right end of $J_n$ and the left end of $J_n^{-1}$ are adjacent in $\G_w$ to $\{e_n\}$ only.  Copies of $I_n$ appears in $\G_w$ only inside $J_{n+1}$, and their right end (the left end of $I_n^{-1}$, respectively) is adjacent only to $e_n$. Their left ends (the right ends of $I_n^{-1}$) are the right ends of $J_{n+1}$, hence they are always adjacent to $e_{n+1}$. One of the letters $b, c, d$, denote it $x$, belongs to $e_n$ but does not belong to $e_{n+1}$. Then $H_{I_n}\le H_{J_{n+1}}$ and $[x, H_{I_n}]\subset H_{J_{n+1}}$. Taking commutators of $H_{J_n}$ with the group generated by $[x, H_{I_n}]$ we show that the alternating group of permutations of the vertices of $J_{n+1}$ is contained in $H_{J_{n+1}}$. But $H_{I_n}$ contains odd permutations, hence $H_{J_{n+1}}=S_{J_{n+1}}$. It follows that $H_{I_{n+1}}$ contains the group $S_{J_n}$, hence $H_{n+1}=S_{J_{n+1}}\times S_{J_n}$.
\end{proof}

We obviously have $H_n<H_{n+1}$. Let $H_\infty$ be the union of the subgroups $H_n$.

\begin{proposition}
The derived subgroup $H_\infty'$ of $H_\infty$ is simple and has index 2 in $H_\infty$.
\end{proposition}

\begin{proof}
The derived subgroup of $H_n$ is the direct product $A_{J_n}\times A_{J_{n-1}}$ of alternating groups of permutations of the vertices of $J_n$ and $J_{n-1}$. It follows that $H_n'$ is the union of the subgroups $A_{J_n}\times A_{J_{n-1}}$. Let $g\in H_n'$ be an arbitrary non-trivial element. Let $(g_1, g_2)\in A_{J_n}\times A_{J_{n-1}}$ be the corresponding element of the direct product of alternating groups. Let $(h_1, h_2)$ be the element of $A_{J_{n+1}}\times A_{J_n}$ representing $g\in H_{n+1}$. The set of vertices of $J_{n+1}$ is a union of a set in a bijection with the set of vertices of $J_n$ and a set in a bijection with the set of vertices of $I_n$. The permutation $h_1$ acts by $g_1$ on the first set and by $g_2$ on the second one. The permutation $h_2$ acts by copies of $g_1$ on two halves of the set of vertices of $I_{n+1}$. Consequently, for all $m\ge n+2$ the both coordinates of $g$ as an element of the direct product $A_{J_m}\times A_{J_{m-1}}$ are non-trivial. Consequently, the normal closure of $g$ in $H_m'$ is equal to $H_m'$. Hence, the normal closure of $g$ in $H_\infty'$ is the whole derived subgroup $H_\infty'$.

If $(k_1, k_2)\in(\Z/2\Z)^2$ is the parity of an element $g$ of $S_{J_n}\times S_{J_{n-1}}\cong H_n$, then the parity of $g$ as an element of $S_{J_{n+1}}\times S_{J_n}\cong H_{n+1}$ is $(k_1, k_1)$, since the permutation in the second coordinate $S_{J_{n-1}}$ is copied twice as a permutation of the second half $I_n$ of $J_{n+1}$. It follows that an element of $S_{J_n}\times S_{J_{n-1}}\cong H_n$ belongs to $H_\infty'$ if and only if its first coordinate is an even permutation. Consequently, $[H_\infty:H_\infty']=2$. An element of $H_\infty$ not belonging to $H_\infty'$ is, for example, $a_2$.
\end{proof}

Denote by $J_\infty$ the inductive limit of the segments $J_n$ with respect to the embedding of $J_n$ to the left part of $J_{n+1}$. Then $J_n$ is a right-infinite chain. Let $\xi$ be its right endpoint. The subshift $\mathcal{S}$ contains the graphs $\Lambda_n=J_\infty^{-1}e_nJ_\infty$ for all $n=0, 1, 2$. The smallest common cover of $\Lambda_n$ is the graph $\Xi$ obtained by taking four copies of $J_\infty$, and connect the copies of $\xi$ together by a Cayley graph of the Klein's four-group $\{1, b, c, d\}$. The four copies of $J_\infty$ in $\Xi$ can be denoted $(J_\infty, x)$ for $x\in\{1, b, c, d\}$ so that the left end of $J_\infty$ in the copy $(J_\infty, x)$ is connected by an edge labeled by $y$ to the left end of $(J_\infty, xy)$. See Figure~\ref{fig:Xi}, where $J_\infty$ is shown above $\Xi$.

\begin{figure}
\includegraphics{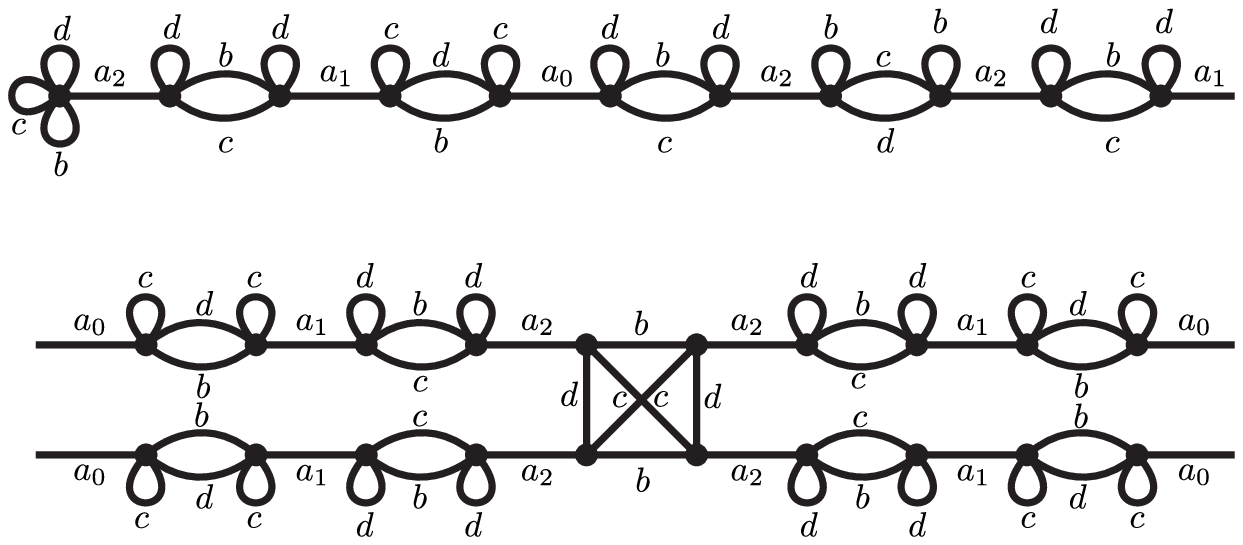}
\caption{Graphs $J_\infty$ and $\Xi$}
\label{fig:Xi}
\end{figure}

Since $\Xi$ covers the graphs $\Lambda_n$, the group $\widehat G$ naturally acts on the set of vertices of $\Xi$, i.e., $\Xi$ is a Schreier graph of $\widehat G$. In fact, $\Xi$ is the Schreier graph of $\widehat G$ with respect to the intersection of the stabilizers of the copies of the left end of $J_\infty$ in the Schreier graphs $\Lambda_n$.

\begin{proposition}
\label{pr:Hinftydescription}
An element $g\in\widehat G$ belongs to $H_\infty$ if and only if it leaves invariant the subsets $(J_\infty, x)$ of $\Xi$.
\end{proposition}

\begin{proof}
If $g\in H_\infty$, then there exists $n$ such that $g\in H_n$. But then $g$ leaves the copies of the segment $J_n, I_n$ invariant, hence it leaves also invariant the subsets $(J_\infty, x)$ of $\Xi$ which are unions of such segments.

Conversely, suppose that $g\in\widehat G$ leaves the subsets $(J_\infty, x)$ invariant. Note that each of the segments $I_n, J_n$ is of the form $J_{n-2}\ldots J_{n-2}^{-1}$. It follows that in the partition $\ldots X_1e_{k_1}X_2e_{k_2}\ldots$ of $\G_w$ into the copies $X_i$ of the segments $I_n, J_n$, the connectors $e_{k_i}$ are surrounded by $J_{n-2}^{-1}e_{k_i}J_{n-2}$. If the length of $g$ is shorter than the length of $J_{n-2}$, then any path describing the action of $g$ on $\G_w$ can be lifted to a path in $\Xi$ so that $e_{k_i}$ is covered by the Cayley graph of $\{1, b, c, d\}$ (i.e., by the central part connecting the copies of $J_\infty$). Since $g$ leaves the parts $(J_\infty, h)$ invariant, the points of $\G_w$ do not cross the connectors $e_{k_i}$ under the action of $g$. Moreover, the action of $g$ on each copy $X_i$ of the segments $I_n, J_n$ will depend only on the isomorphism class of $X_i$. It follows that $g\in H_\infty$.
\end{proof}

\begin{theorem}
\label{th:virtsimple}
The group $\widehat G$ is virtually simple.
\end{theorem}

\begin{proof}
Let $A=\langle{H_\infty'}^{\widehat G}\rangle$ be the normal closure of $H_\infty'$ in $\widehat G$. We will prove that it is a simple subgroup of finite index.

Let us prove first that $A$ has finite index.
For a vertex $v\in J_\infty$ and an element $g\in\widehat G$, denote by $\tau_g(v)$ the element of $\{1, b ,c, d\}$ such that the lift of the action of $G$ to $\Xi$ moves $(v, 1)$ to $(g(v), \tau_g(v))$. 
We have $\tau_g(v)=1$ for all but finitely many vertices of $J_\infty$. We have \begin{equation}
\label{eq:taucocycle}
\tau_{g_1g_2}(v)=\tau_{g_1}(g_2(v))\tau_{g_2}(v)
\end{equation}
for all $g_1, g_2\in\widehat G$ and all vertices $v$.

Denote by $\phi(g)$ the product of the values of $\tau_g(v)$ for all vertices of $J_\infty$. It follows from~\eqref{eq:taucocycle} and the fact that $\{1, b, c, d\}$ is commutative, that $\phi:\widehat G\arr\{1, b, c, d\}$ is a homomorphism. It is also easy to check that $\phi(b)=b, \phi(c)=c$, and $\phi(d)=d$, so $\phi$ is an epimorphism.

We have $H_\infty<\ker\phi$. Moreover, by Proposition~\ref{pr:Hinftydescription}, an element $g\in\widehat G$ belongs to $H_\infty$ if and only if $\tau_g(v)=1$ for all vertices $v$ of $J_\infty$

It is checked directly that $\phi_{[a_2, b]}(v)$ is equal to 1 for all vertices $v$ except for two of them ($\xi$ and $ba_2(\xi)$) where it is equal to $b$. Similarly, $\phi_{[a_2, c]}$ is equal to $c$ on two vertices and to $1$ everywhere else. 
Since $H_\infty'$ is 2-transitive on $J_\infty$, it follows that every element of $\ker\phi$ can be written as a product of an element of $H_\infty$ and elements of the form $[a_2, b]^h$ and $[a_2, c]^h$ for $h\in H_\infty'$.  Note that $[a_2, b], [a_2, c]\in A$, since $a_2\in H_\infty'$.
As $[H_\infty:H_\infty']=2$, this implies that $A$ is a subgroup of index at most $2$ in $\ker\phi$. Consequently, $A$ is a subgroup of finite index in $\widehat G$.

Let us show that $A$ is simple. Suppose that $N\lhd A$ is a non-trivial proper normal subgroup of $A$. Let $g\in N\setminus\{1\}$.
The element $g$ moves a vertex $v$ of $J_\infty^{-1}e_1 J_\infty$. There exists $n$ such that $v, g(v)$ belong to the central segment $J_n^{-1}e_1J_n$ of $J_\infty^{-1}e_1J_\infty$, and the distance from $v$ to the boundary of $J_n^{-1}e_1J_n$ is greater than the length of $g$.
We can find  an isomorphic copy of $J_n^{-1}e_1J_n$ inside $J_m$ for some large $m$ and an element $h\in H_m'$ moving $v$ but fixing pointwise all elements of the $|g|$-neighborhood of the boundary of $J_m$. Then the commutator $[h, g]$ is non-trivial and belongs to $H_m'$. Since $H_\infty'$ is simple, it follows that $N\ge H_\infty'$.

It remains to show that the normal closure of $H_\infty'$ in $\widehat G$ is equal to the normal closure of $H_\infty'$ in $A$. The group $H_\infty'$ is generated by the set of permutations $g\in\widehat G$ such that the $\langle g\rangle$-orbits are of lengths 1 and 3, and $g$ preserves the set of vertices of $J_\infty$ in $\Xi$. Moreover, for any $D$ we can take as a generating set a subset $C$ of this set such that for every $g\in C$ every two $\langle g\rangle$-orbits of length 3 are on distance at least $D$ from each other.

It follows from the arguments from the proof that $A$ is a subgroup of finite index in $\widehat G$ that a conjugate of an element of $C$ by an element of $\widehat G$ is equal to a conjugate by an element of $A$. Consequently, the group generated by ${H_\infty'}^A$ is equal to the group generated by ${H_\infty'}^{\widehat G}$.
\end{proof}

The intersection of $G$ with the simple finite index subgroup of $\widehat G$ has finite index in $G$. It is known that every subgroup of finite index in $G$ has a subgroup isomorphic to $G$ (see, for example~\cite[Section~12]{grigorchuk:branch}). Consequently, the simple finite index subgroup of $\widehat G$ contains an isomorphic copy the Grigorchuk group. So, we have embedded the Grigorchuk group into a simple finitely generated torsion group of intermediate growth $\widehat G$. The group $\widehat G$ is torsion and of intermediate growth by Theorem~\ref{th:intermediategrowth}.

\section{Uncountably many growth types}
We construct in this section, using the techniques of defining groups by their Schreier graphs, an uncountable family of virtually simple groups containing continuum of pairwise different growth types. 

Our set of generators will be $S=\{a_0, a_1, a_2, x, y, b, c, d\}$. We will use the same notation
\[e_1=\{b, c\},\quad e_2=\{d, b\},\quad e_3=\{c, d\}\]
as for the Grigorchuk group and the virtually simple group $\widehat G$.

Start with
\[I_0=\{a_0\}\{a_1\},\qquad J_0=\{a_0\}\{a_2\}.\]

Let $\alpha=(\alpha_0, \alpha_1, \ldots)$ be a sequence of symbols $\sigma$ or non-trivial elements of the free monoid $\{x, y\}^*$. Define the sequence of pairs of segments, associated with $\alpha$, by the following rule.

If $\alpha_n=\sigma$, then
\begin{eqnarray*}
I_{n+1}&=&J_ne_nJ_n^{-1},\\
J_{n+1}&=& J_n\{x\}I_n\{y\}I_n\{y\}J_n^{-1}.
\end{eqnarray*}
If $\alpha_n=t_1t_2\ldots t_m\in\{x, y\}^*$, then
\begin{eqnarray*}
I_{n+1}&=& J_ne_nJ_n^{-1},\\
J_{n+1}&=&J_ne_1I_n\{t_1\}I_n\{t_2\}I_n\{t_3\}\ldots\{t_m\}J_n^{-1}.
\end{eqnarray*}

Let $\mathcal{S}_\alpha$ be the subshift defined by the set of $\{I_n, J_n\}_{n\ge 0}$, and let $G_\alpha$ be the group defined by $\mathcal{S}_\alpha$.

Note that since every two connectors $e_n$ are disjoint, the subgroup $\{1, b, c, d\}$ of $G_\alpha$ is isomorphic to the Klein's four group for every $\alpha$.

\begin{proposition}
\label{pr:virtsimpalpha}
The group $G_\alpha$ is virtually simple for every sequence $\alpha$. More precisely, the derived subgroup $G_\alpha'$ is simple and $G_\alpha/G_\alpha'\cong (\Z/2\Z)^4$.
\end{proposition}

\begin{proof}
For every $n\ge 0$ every element of the subshift $\mathcal{S}_\alpha$ is obtained by concatenating isomorphic copies of the segments $I_n$, $J_n$ using the connectors $e_1, e_2, e_3$, $\{x\}$, and $\{y\}$. It is easy to see that the partition into the copies of the segments $I_0$ and $J_0$ is unique. Since the segment $I_{n+1}$ is the only segment of the form $J_ne_nJ_n^{-1}$, it follows by induction that the partition into the copies of $I_n$ and $J_n$ is unique for every $n\ge 0$.

Each segment $I_n$ consists of two copies of $J_{n-1}$. Therefore, we get a canonical partition of each graph $\G\in\mathcal{S}_\alpha$ into isomorphic copies of $J_n$ and $J_{n-1}$.

Similarly to Section~\ref{s:simpleGrigorchuk}, denote for $n\ge 1$ by $H_n$ the subgroup of all elements $g\in G_\alpha$ preserving the set of vertices of each copy of $J_n$ and $J_{n-1}$ of the partition and commuting with all isomorphisms between the copies of $J_n$ and between the copies $J_{n-1}$.  The group $H_n$ is naturally identified with a subgroup of the direct product $S_{J_n}\times S_{J_{n-1}}$ of symmetric groups on the sets of vertices of $J_n$ and $J_{n-1}$.
Denote by $H_{J_n}$ and $H_{I_n}$ the intersections of $H_n$ with $S_{J_n}\times\{1\}$ and $\{1\}\times S_{J_{n-1}}$. 

Note that $a_1, [a_0, a_1]$ generate a subgroup of $G_\alpha$ isomorphic to $S_3$, preserving the vertices of each copy of $I_0$, acting trivially on the copies of $J_0$, and commuting with isomorphisms between the copies of $I_0$. Similarly, $a_2, [a_0, a_2]$ generate a subgroup of $G_\alpha$ acting as the full symmetric group on the sets of vertices of the copies of $J_0$.

Let us prove at first that $H_n$ contains the direct product $A_{J_n}\times A_{J_{n-1}}$ of the groups of even permutations of the sets of vertices of $J_n$ and $J_{n-1}$, i.e., that $H_{J_n}$ contains the alternating group on $J_n$ and $H_{I_n}$ contains the alternating group on $J_{n-1}$. Let us prove this by induction. Both the inductive step and the base case $n=1$ will be proved at the same time.

Let $v_1, v_2, \ldots, v_m$ be the vertices of $I_n$ listed in the order they appear in the segment. Then $H_n$ contains the permutation $h=(v_1, v_2)(v_{m-1}, v_m)$ for $n\ge 1$ or the permutation $(v_1, v_2)$ for $n=0$, acting identically on the copies of $J_n$. Consider the commutator $[x, h]$ if $\alpha_n=\sigma$ or the commutator $[b, h]$ if $\alpha_n\ne\sigma$. This commutator will be a cycle of length 3 on the set of vertices of $J_{n+1}$ and will act identically on the vertices of $I_{n+1}$. Conjugating it by elements of $H_{J_n}$, by elements of $H_{I_n}$, and by $x, y$, $b, c, d$, we will get enough 3-cycles to generate $A_{J_{n+1}}\le H_{J_{n+1}}$. The group $H_{I_{n+1}}$ will contain $H_{J_n}$ by the inductive assumption (or by the fact that $\langle a_2, [a_0, a_2]\rangle$ is $S_3$). This finishes the inductive argument.

Let $h_1\in S_{J_n}, h_2\in S_{J_{n+1}}$ be arbitrary permutations. Since there are even numbers of copies of $J_n$ and $J_{n-1}$ in $J_{n+1}$, they will induce an even permutation of the set of vertices of $J_{n+1}$. By the same argument, they will induce an even permutation of the set of vertices of $J_{n+2}$. It follows that the corresponding permutation belongs to $A_{J_{n+1}}\times A_{J_{n+2}}\le H_{n+2}$. Consequently, $S_{J_n}\times S_{J_{n-1}}$ is contained in $G_\alpha$, hence $H_n$ coincides with the whole group $S_{J_n}\times S_{J_{n-1}}$.

Denote by $H_\infty$ the union of the groups $H_n$. Note that we have shown that $H_\infty$ is isomorphic to the direct limit of the direct products of the alternating groups $A_{J_n}\times A_{J_{n-1}}$, which implies that $H_\infty$ is perfect.

Let $J_\infty$ be the limit of $J_n$ with respect to the embedding of $J_n$ to the left end of $J_{n+1}$. It follows from the recursions defining the segments $I_n$ and $J_n$ that the subshift $\mathcal{S}_\alpha$ contains the graphs $J_\infty^{-1}zJ_\infty$ for each $z\in\{\{x\}, \{y\}, e_1, e_2, e_3\}$, since the corresponding finite segments $J_n^{-1}zJ_n$ are sub-segments of $I_m$ and $J_m$ for all $m\ge n+3$. The same arguments as in the proof of Proposition~\ref{pr:Hinftydescription} show that an element $g\in G_\alpha$ belongs to $H_\infty$ if and only if it leaves invariant the sets of vertices of $J_\infty$ and $J_\infty^{-1}$ in the graphs $J_\infty^{-1}zJ_\infty$ for every $z\in\{\{x\}, \{y\}, e_1, e_2, e_3\}$. Note that $x$, $y$, $b$, $c$, $d$ generate a group $K$ isomorphic to $(\Z/2\Z)^4$, as they commute with each other. Similarly to the proof of Theorem~\ref{th:virtsimple}, define $\tau_g(v)$ for a vertex $v$ of $J_\infty$ to be the product $h_1h_2h_3\in K$ of the elements $h_1\in\langle x\rangle$, $h_2\in\langle y\rangle$, $h_3\in\{1, b, c, d\}$, where $h_1$ is equal to $x$ if $v$ is moved by $g$ from $J_\infty$ to $J_\infty^{-1}$ in $J_\infty^{-1}\{x\}J_\infty$ and to $1$ otherwise, $h_2$ is equal to $y$ if $v$ is moved by $g$ from $J_\infty$ to $J_\infty^{-1}$ in $J_\infty^{-1}\{y\}J_\infty$ and to $1$ otherwise, and $h_3$ describes to which branch $(J_\infty, h)$ the vertex $v$ is moved by $g$ in the graph $\Xi$ obtained by gluing the copies of $J_\infty$ together along the endpoint of $J_\infty$ by the Cayley graph of $\{1, b, c, d\}$.

Then $g\in H_\infty$ if and only if $\tau_g$ is constant $1$. Let $\phi(g)$ be the product of the values of $\tau_g(v)$ over all vertices $v$ of $J_\infty$. Then, by the same arguments as in the proof of Theorem~\ref{th:virtsimple}, $\phi:G_\alpha\arr K$ is an epimorphism and its kernel is equal to the normal closure of $H_\infty$ in $G_\alpha$ (we are in a better situation here than for Theorem~\ref{th:virtsimple}, since $H_\infty'=H_\infty$ now).

The same argument as in the proof of Theorem~\ref{th:virtsimple} show that for every non-trivial element $g\in\ker\phi$ the normal closure of $g$ in $\ker\phi$ contains $H_\infty$ and that the normal closure of $H_\infty$ in $\ker\phi$ and in $G_\alpha$ coincide, which proves that $\ker\phi$ is simple. Note that this implies that $\ker\phi=G_\alpha'$, since $K$ is commutative.
\end{proof}

\begin{proposition}
\label{pr:intermediatelong}
Let $\alpha=(\alpha_1, \alpha_2, \ldots)$ be a sequence such that $\alpha_n=\sigma$ for all $n$ big enough. Then there exist $C_1, C_2$ such that the growth of $G_\alpha$ satisfies $\gamma_{G_\alpha}(n)\le \exp(C_1/\exp(C_2\sqrt{\log n}))$ for all $n\ge 1$.

For every finite sequence $(\alpha_1, \alpha_2, \ldots, \alpha_k)$ of symbols $\sigma$ and elements of $\{x, y\}^*$ and for every $R\ge 1$ there exists $n$ such that for every sequence \[\alpha=(\alpha_1, \alpha_2, \ldots, \alpha_k, \underbrace{\sigma, \sigma, \ldots, \sigma}_{\text{$n$ times}}, \alpha_{k+n+1}, \alpha_{k+n+2}, \ldots)\]
the ball of radius $R$ in the Cayley graph of $G_\alpha$ is isomorphic to the ball of radius $R$ of the group $G_{(\alpha_1, \alpha_2, \ldots, \alpha_k, \sigma, \sigma, \sigma, \ldots)}$.
\end{proposition}

\begin{proof}
The first statement follows directly from Theorem~\ref{th:intermediategrowth}. 

Let us prove the second statement. For every $(\alpha_1, \alpha_2, \ldots, \alpha_k)$ and $R$ there exists $n$ such that every for every segment $I_m, J_m$ for $m\ge k+n$ has length more than $R$. 

Note that the set of segments of length $R$ in $I_{m+1}$ and $J_{m+1}$ do not depend then of the symbol $\alpha_m$. Consequently, for every \[\alpha=(\alpha_1, \alpha_2, \ldots, \alpha_k, \underbrace{\sigma, \sigma, \ldots, \sigma}_{\text{$n$ times}}, \alpha_{k+n+1}, \alpha_{k+n+2}, \ldots)\] the set of segments of length $R$ in the graphs belonging to $\mathcal{S}_\alpha$ coincides with the set of segments of length $R$ in the graphs belonging to $\mathcal{S}_{\alpha'}$ for for $\alpha'=(\alpha_1, \alpha_2, \ldots, \alpha_k, \sigma, \sigma, \ldots)$. It follows that an equality $g_1=g_2$  for products of length $\le R$ of the generators $x, y, b, c, d, a_0, a_1, a_2$ are true or false in $G_\alpha$ and $G_{\alpha'}$ at the same time, i.e., the groups $G_\alpha$ and $G_{\alpha'}$ have isomorphic balls of radius $R$ in their Cayley graphs.
\end{proof}

\begin{proposition}
\label{pr:exponentiallong}
For every finite sequence $(\alpha_1, \alpha_2, \ldots, \alpha_{k-1})$ there exists $M$ such that for every $N$ there exists a word $w\in\{x, y\}^*$ such that for every infinite sequence $\alpha=(\alpha_1, \alpha_2, \ldots, \alpha_{k-1}, w, \alpha_{k+1}, \alpha_{k+2}, \ldots)$, the growth of the group $G_\alpha$ satisfies
\[\gamma_{G_\alpha}(Mn)\ge 2^n\]
for all $n=1, \ldots, N$.
\end{proposition}

\begin{proof}
Consider the segment $I_k$ defined by $(\alpha_1, \alpha_2, \ldots, \alpha_{k-1})$, and let $g$ be the product of the labels along a simple path from the left to the right endpoints of $I_k$, so that $g$ has length equal to the length of $I_k$ and moves it left end to the right end. 

Let $w\in\{x, y\}^*$ be a word containing all words $v\in\{x, y\}^*$ of length $\le N$ as subwords. For every word $v=t_1t_2\ldots t_n\in\{x, y\}^n$ consider the corresponding element $g_v=t_ng\cdot t_{n-1}g\cdots t_1g$ of $G_\alpha$ for $\alpha=(\alpha_1, \alpha_2, \ldots, \alpha_{k-1}, w, \alpha_{k+1}, \alpha_{k+2}, \ldots)$. 

Since $w$ contains $v=t_1t_2\ldots t_n\in\{x, y\}^n$, the segment $I_{k+1}$ will contain the sub-segment $Z_v=I_k\{t_1\}I_k\{t_2\}\ldots I_k\{t_n\}$. The element $g_v=t_ng\cdot t_{n-1}g\cdots t_1g$ will move the left end of $Z_v$ to its right end and the length of $g_v$ is equal to the length of $Z_v$. Any other element $g_{v'}$ for $v'\in\{x, y\}^*$ of length $\le n$ will move the left end of $Z_v$ to another vertex, since any geodesic path connecting the ends of $Z_v$ must have length $n$ and contain exactly the same edges $t_1, t_2, \ldots, t_n$ on the same places as in the geodesic path corresponding to $g_v$. It follows that all elements $g_v$ for $v\in\{x, y\}^n$, for $n\le N$, are pairwise different. There are $2^n$ of them, and their length is $n(|I_k|+1)$. It follows that $M=|I_k|+1$ and the chosen $w$ will satisfy the conditions of the proposition.
\end{proof}

\begin{theorem}
There exist uncountably many pairwise different growth types of simple finitely generated groups.
\end{theorem}

\begin{proof}
After we proved Propositions~\ref{pr:intermediatelong} and~\ref{pr:exponentiallong}, the proof of the theorem is similar to the proof of~\cite[Theorem~7.2]{grigorchuk:growth_en}.

Choose a sequence $C_1, C_2, \ldots$ of positive integers converging to infinity. 
We will define for every sequence $\rho=(r_1, r_2, \ldots)\in\{\zero, \one\}^\infty$ a sequence $\alpha_\rho$ such that there are uncountably many different growth types among the groups $G_{\alpha_\rho}$. The sequence $\alpha_\rho$ will be the limit of finite sequences $\alpha_{(r_1, r_2, \ldots, r_n)}$ so that $\alpha_{(r_1, r_2, \ldots, r_{n+1})}$ is a continuation of $\alpha_{(r_1, r_2, \ldots, r_n)}$. At the same time we will define a sequence $R_1, R_2, \ldots$ of positive integers.

Using Proposition~\ref{pr:intermediatelong} and Proposition~\ref{pr:exponentiallong}, we can find $R_1$,  $\alpha_{\zero}=(\sigma, \sigma, \ldots, \sigma)$, and $\alpha_{\one}=(w)$ for $w\in\{x, y\}^*$, such that if $\gamma_0(n)$ is the growth of $G_\alpha$ for any $\alpha$ beginning with $\alpha_\zero$, and $\gamma_1(n)$ is the growth of $G_\alpha$ for any $\alpha$ beginning with $\alpha_\one$, then we have
\[\gamma_0(C_1R_1)\le\gamma_1(R_1).\]

Suppose that we have defined $R_n$ and the sequences $\alpha_v$ for all sequences $v\in\{\zero, \one\}^n$ of length $n$. Then by Propositions~\ref{pr:intermediatelong} and~\ref{pr:exponentiallong}, there exists a number $k$,  a word $w\in\{x, y\}^*$, and a number $R_{n+1}>R_n$ such that if $\gamma_0$ is the growth function of any group $G_\alpha$, where $\alpha$ starts with $\alpha_v$ for $v\in\{\zero, \one\}^n$ followed by $\underbrace{\sigma, \sigma, \ldots, \sigma}_{\text{$k$ times}}$, and $\gamma_1$ is the growth function of any group $G_{\alpha'}$ for $\alpha'$ starting with $\alpha_u$ for $u\in\{\zero, \one\}^n$ followed by $w$, then we have
\[\gamma_0(C_{n+1}R_{n+1})\le\gamma_1(R_{n+1}).\]

This will give us an inductive definition of the groups $G_{\alpha_\rho}$ for $\rho\in\{\zero, \one\}^\infty$ and of the sequence $R_n$. Suppose now that $\rho_1, \rho_2\in\{\zero, \one\}^\infty$ are two sequences which differ in infinitely many coordinates. Suppose that the growth functions $\gamma_i$ of $G_{\alpha_{\rho_i}}$ are equivalent. Then there exists $C>1$ such that $\gamma_2(R)<\gamma_1(CR)$ and $\gamma_1(R)<\gamma_2(CR)$ for all $R\ge 1$. Since $\rho_1$ and $\rho_2$ are different in infinitely many coordinates, there exists $n$ such that $C_n>C$ and the $n$th coordinates of $\rho_1$ and $\rho_2$ are different. Then we will have either $\gamma_1(C_nR_n)\le\gamma_2(R_n)<\gamma_1(CR_n)$ or $\gamma_2(C_nR_n)\le\gamma_1(R_n)<\gamma_2(CR_n)$, which is a contradiction.

It follows that the sets of groups with equivalent growth functions in the set of groups $\{G_{\alpha_\rho}\;:\;\rho\in\{\zero, \one\}^\infty\}$ are at most countable, hence we have continuum pairwise different growth types in this set. Each group in this set is virtually simple by Proposition~\ref{pr:virtsimpalpha}. Since a finite index subgroup has the same growth type as the group, it follows that there is a continuum of different growth types of simple finitely generated groups.
\end{proof}


\begin{thebibliography}{MOW19}

\bibitem[BE14]{bartholdiErschler:imbeddings}
Laurent Bartholdi and Anna Erschler, \emph{Imbeddings into groups of
  intermediate growth}, Groups Geom. Dyn. \textbf{8} (2014), no.~3, 605--620.

\bibitem[BG00]{bgr:spec}
Laurent Bartholdi and Rostislav~I. Grigorchuk, \emph{On the spectrum of {H}ecke
  type operators related to some fractal groups}, Proceedings of the Steklov
  Institute of Mathematics \textbf{231} (2000), 5--45.

\bibitem[BR10]{bertherigo:combinatoricsautomata}
Val\'{e}rie Berth\'{e} and Michel Rigo (eds.), \emph{Combinatorics, automata
  and number theory}, Encyclopedia of Mathematics and its Applications, vol.
  135, Cambridge University Press, Cambridge, 2010.

\bibitem[CR10]{capraceremy:building}
Pierre-Emmanuel Caprace and Bertrand R\'{e}my, \emph{Non-distortion of twin
  building lattices}, Geom. Dedicata \textbf{147} (2010), 397--408.

\bibitem[Den32]{denjoy:curbes}
Arnaud Denjoy, \emph{Sur les courbes d\'efinies par les \'equations
  diff\'erentielles \`a la surface du tore}, J. Math. Pures Appl. (9)
  \textbf{11} (1932), 333--375.

\bibitem[Gri80]{grigorchuk:80_en}
Rostislav~I. Grigorchuk, \emph{On {Burnside's} problem on periodic groups},
  Functional Anal.\ Appl. \textbf{14} (1980), no.~1, 41--43.

\bibitem[Gri85]{grigorchuk:growth_en}
\bysame, \emph{Degrees of growth of finitely generated groups and the theory of
  invariant means}, Math.\ USSR Izv. \textbf{25} (1985), no.~2, 259--300.

\bibitem[Gri00]{grigorchuk:branch}
\bysame, \emph{Just infinite branch groups}, New Horizons in pro-$p$ Groups
  (Aner Shalev, Marcus~P.~F. {du Sautoy}, and Dan Segal, eds.), Progress in
  Mathematics, vol. 184, Birkh{\"a}user Verlag, Basel, 2000, pp.~121--179.

\bibitem[JM13]{juschenkomonod}
Kate Juschenko and Nicolas Monod, \emph{Cantor systems, piecewise translations
  and simple amenable groups}, Ann. of Math. (2) \textbf{178} (2013), no.~2,
  775--787.

\bibitem[Mat12]{matui:etale}
Hiroki Matui, \emph{Homology and topological full groups of \'etale groupoids
  on totally disconnected spaces}, Proc. Lond. Math. Soc. (3) \textbf{104}
  (2012), no.~1, 27--56.

\bibitem[MOW19]{minasyanosinwitzel}
Ashot Minasyan, Denis Osin, and Stefan Witzel, \emph{Quasi-isometric diversity
  of marked groups}, (preprint arXiv:1911.01137), 2019.

\bibitem[Nek18]{nek:burnside}
Volodymyr Nekrashevych, \emph{Palindromic subshifts and simple periodic groups
  of intermediate growth}, Annals of Math. \textbf{187} (2018), no.~3,
  667--719.

\bibitem[Nek19]{nek:fullgr}
\bysame, \emph{Simple groups of dynamical origin}, Ergodic Theory and Dynamical
  Systems \textbf{39} (2019), no.~3, 707--732.

\bibitem[Nie28]{nielsen:denjoy}
Jakob Nielsen, \emph{Om topologiske afbildninger af en jordankurve paa sig
  selv}, Matematisk Tidsskrift. B (1928), 39--46.

\bibitem[SWZ19]{SkipperZaremsky}
Rachel Skipper, Stefan Witzel, and Matthew C.~B. Zaremsky, \emph{Simple groups
  separated by finiteness properties}, Invent. Math. \textbf{215} (2019),
  no.~2, 713--740.

\bibitem[Vor12]{vorob:schreiergraphs}
Yaroslav Vorobets, \emph{Notes on the {Schreier} graphs of the {Grigorchuk}
  group}, Dynamical systems and group actions (L.~Bowen et~al., ed.), Contemp.
  Math., vol. 567, Amer. Math. Soc., Providence, RI, 2012, pp.~221--248.

\end{thebibliography}
\end{document}